\documentclass[11pt]{article}
\usepackage[utf8]{inputenc}
\usepackage[margin=1in]{geometry}
\usepackage{indentfirst}
\usepackage{amsmath, amssymb, amsthm}
\usepackage{mathrsfs}
\usepackage{setspace}
\usepackage{wasysym}
\usepackage{enumerate}
\usepackage[colorlinks=true,linkcolor=purple, citecolor=blue,urlcolor=magenta]{hyperref}
\usepackage{xcolor}
\usepackage{graphicx}
\usepackage{tikz}
\usepackage{float}
\usepackage{booktabs}
\usepackage{pgfplots}

\newtheorem{theo}{Theorem}[section]
\newtheorem{lem}{Lemma}[section]

\newcommand{\ol}{\overline}

\numberwithin{equation}{section}

\begin{document}

\title{Inverse Logarithmic Coefficients and Associated Sharp Estimates for the Apple-Like Convex Subclass $\mathcal{C}_{\mathcal{AP}}$}

\author{Pradip Das\thanks{Department of Mathematics, Raiganj University, Raiganj, West Bengal-733134, India. Email: pradipsmath@gmail.com} 
\and 
Nabadwip Sarkar\thanks{Amity School of Applied Sciences, Amity University Mumbai, Panvel, Navi Mumbai, Maharashtra--410206, India. Email: naba.iitbmath@gmail.com}}

\date{}

\maketitle

\vspace{1mm}

\begin{abstract}
This paper addresses coefficient problems for the apple-like convex subclass $\mathcal{C}_{\mathcal{AP}}$, defined by the subordination relation $1+zf''(z)/f'(z) \prec e^z\sqrt{1+z}$. We determine sharp bounds for the initial inverse logarithmic coefficients $\Gamma_1$, $\Gamma_2$, $\Gamma_3$, and the consecutive modulus difference $|\Gamma_2|-|\Gamma_1|$. We also obtain the sharp upper bound for the second-order inverse logarithmic Hankel determinant, characterize the generalized Fekete--Szeg\H{o} functional for all real parameters, and compute sharp bounds for the third-order Hermitian--Toeplitz determinant. Extremal functions are given for each case.
\end{abstract}



\section{Introduction}
Let $\mathcal{H}$ denote the class of analytic functions in the open unit disk $\mathbb{D}:=\{z\in\mathbb{C}:|z|<1\}$. Let $\mathcal{A}$ be the subclass of $\mathcal{H}$ consisting of functions normalized by $f(0)=0$ and $f'(0)=1$, and let $\mathcal{S}$ denote the class of all univalent functions $f\in\mathcal{A}$ having the Taylor series expansion
\begin{equation}\label{eq1}
f(z)=z+\sum_{n=2}^{\infty}a_nz^n, \quad z\in\mathbb{D}.
\end{equation}
We denote by $\mathcal{S}^{\ast}$ and $\mathcal{C}$ the standard subclasses of starlike and convex functions in $\mathcal{S}$, respectively. Analytically, $f \in \mathcal{S}^{\ast}$ if and only if $\Re(zf'(z)/f(z))>0$ for $z\in\mathbb{D}$, and $f \in \mathcal{C}$ if and only if $\Re(1+zf''(z)/f'(z))>0$ for $z\in\mathbb{D}$. The classical Alexander relation states that $f\in\mathcal{C}$ if and only if $zf'\in\mathcal{S}^{\ast}$ (see Duren \cite{PLD1} and Goodman \cite{G}).

An analytic function $f$ is subordinate to $g$ in $\mathbb{D}$, written $f\prec g$, if there exists a Schwarz function $\omega$ such that $f(z)=g(\omega(z))$ for $z\in\mathbb{D}$. If $g$ is univalent in $\mathbb{D}$, then $f\prec g$ is equivalent to $f(0)=g(0)$ and $f(\mathbb{D})\subset g(\mathbb{D})$.

Ma and Minda \cite{MM} used subordination to introduce a unified framework for subclasses of starlike functions, defined by
\[
\mathcal{S}^{\ast}(\psi) = \left\{ f\in\mathcal{S}: \frac{zf'(z)}{f(z)} \prec \psi(z), \quad z\in\mathbb{D}\right\},
\]
where $\psi$ is analytic, univalent, maps $\mathbb{D}$ onto a domain symmetric with respect to the real axis, and satisfies $\Re(\psi(z))>0$, $\psi(0)=1$, and $\psi'(0)>0$.

Various subordinate functions $\psi$ have been introduced to explore distinct geometric configurations \cite{MM, Ravichandran2015, RavichandranVerma2017}. Recently, Sana, Saliu, and Riaz \cite{SanaSaliuRiaz} defined subclasses associated with the apple-like function $\psi_{\mathcal{AP}}(z)=e^{z}\sqrt{1+z}$. A function $f\in\mathcal{A}$ belongs to the apple-like convex class $\mathcal{C}_{\mathcal{AP}}$ if $1+\frac{zf''(z)}{f'(z)}\prec \psi_{\mathcal{AP}}(z)$ ($z\in\mathbb{D}$). By the Alexander relation, $f\in\mathcal{C}_{\mathcal{AP}}$ if and only if $zf'\in\mathcal{S}_{\mathcal{AP}}^{\ast}$. The function $\psi_{\mathcal{AP}}$ maps $\mathbb{D}$ onto an apple-shaped region and satisfies the Ma--Minda conditions, making $\mathcal{C}_{\mathcal{AP}}$ a subclass of the classical convex family.
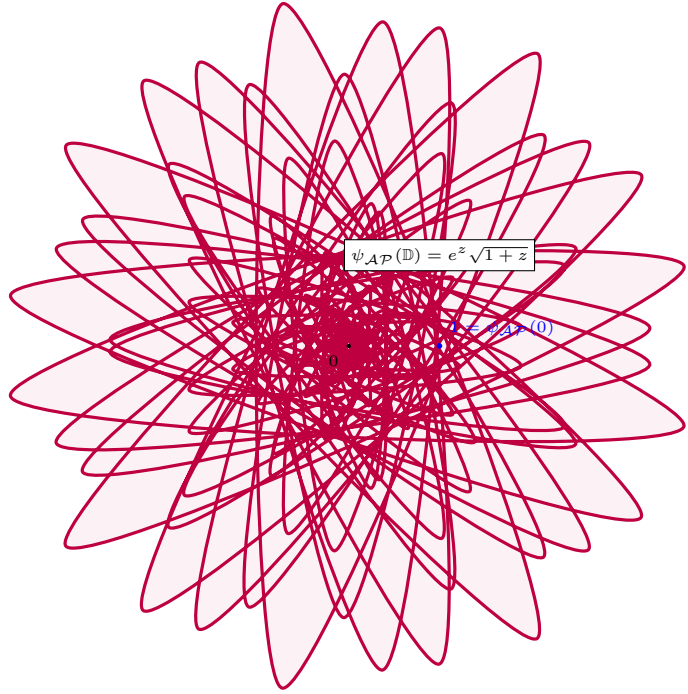
\begin{figure}[htbp]
\centering
\begin{tikzpicture}[scale=1.2] % Reduced scale to make the graphic compact
    % Draw thin mathematical alignment grid lines with scaled boundaries
    \draw[very thin, color=gray!15, step=0.5] (-0.5,-1.5) grid (2.2,1.5);
    
    % Real and Imaginary coordinate axes shortened for tight fitting
    \draw[->, thick] (-0.5,0) -- (2.3,0) node[right] {\footnotesize $\text{Re}$};
    \draw[->, thick] (0,-1.4) -- (0,1.4) node[above] {\footnotesize $\text{Im}$};
    
    % Parametric plot rendering the exact compact Apple-shaped boundary region psi(D)
    \draw[color=purple, line width=1.2pt, smooth, fill=purple!5, domain=-175:175, samples=150] 
        plot ({exp(cos(\x)) * cos(\x r + 0.5*\x r) * sqrt(2*cos(0.5*\x))}, 
              {exp(cos(\x)) * sin(\x r + 0.5*\x r) * sqrt(2*cos(0.5*\x))});
              
    % Critical reference nodes scaled proportionally
    \filldraw[black] (0,0) circle (0.5pt) node[below left, font=\tiny] {$0$};
    \filldraw[blue] (1,0) circle (0.7pt) node[above right, font=\tiny] {$1 = \psi_{\mathcal{AP}}(0)$};
    
    % Domain descriptor label box positioned carefully for smaller dimensions
    \node[draw, fill=white, font=\tiny, inner sep=2.5pt] at (1.0, 1.0) {$\psi_{\mathcal{AP}}(\mathbb{D}) = e^z\sqrt{1+z}$};
\end{tikzpicture}
\caption{The asymmetric, apple-shaped boundary image domain under the mapping $\psi_{\mathcal{AP}}(z) = e^z\sqrt{1+z}$ in the complex plane.}
\label{fig:apple_domain_region}
\end{figure}
\subsection{Inverse Logarithmic Coefficients}
For $f\in\mathcal{S}$, the inverse logarithmic coefficients $\Gamma_n$ are defined by Ponnusamy et al. \cite{Ponnusamy2018} via the expansion of the inverse function $f^{-1}$:
\[
\log\frac{f^{-1}(w)}{w} = 2\sum_{n=1}^{\infty}\Gamma_n w^n, \quad |w|<\frac{1}{4}.
\]
Direct computation yields the first three inverse logarithmic coefficients in terms of the Taylor coefficients of $f$:
\begin{equation}\label{IG1}
\begin{cases}
\Gamma_1 = -\frac{1}{2} a_2, \\
\Gamma_2 = -\frac{1}{2} a_3 + \frac{3}{4} a_2^2, \\
\Gamma_3 = -\frac{1}{2} \left( a_4 - 4a_2 a_3 + \frac{10}{3} a_2^3 \right).
\end{cases}
\end{equation}
For the full class $\mathcal{S}$, Ponnusamy et al. \cite{Ponnusamy2018} proved that $|\Gamma_n| \le \frac{1}{2n}\binom{2n}{n}$ ($n\in\mathbb{N}$), with equality only for the Koebe function and its rotations.

Determinants with logarithmic coefficient entries frequently serve to describe structural properties in geometric function theory. Kowalczyk and Lecko \cite{12,15} studied the classical Hankel determinant $H_{q,n}(F_f/2)$ for logarithmic coefficients. Its inverse analogue, $H_{q,n}(F_{f^{-1}}/2)$, is defined using the inverse logarithmic coefficients $\Gamma_n$ of $f^{-1}$ as entries (see \cite{16}):
\[
H_{q,n}\left(F_{f^{-1}}/2\right)
=
\begin{vmatrix}
\Gamma_n & \Gamma_{n+1} & \cdots & \Gamma_{n+q-1} \\
\Gamma_{n+1} & \Gamma_{n+2} & \cdots & \Gamma_{n+q} \\
\vdots & \vdots & \ddots & \vdots \\
\Gamma_{n+q-1} & \Gamma_{n+q} & \cdots & \Gamma_{n+2(q-1)}
\end{vmatrix}.
\]

In this paper, we determine sharp bounds for the following functionals over the class $\mathcal{C}_{\mathcal{AP}}$:
\begin{itemize}
\item The initial inverse logarithmic coefficients $\Gamma_1$, $\Gamma_2$, and $\Gamma_3$.
\item The second-order inverse logarithmic Hankel determinant $H_{2,1}(F_{f^{-1}}/2) = \Gamma_1\Gamma_3-\Gamma_2^2$.
\item The consecutive coefficient difference $|\Gamma_2|-|\Gamma_1|$.
\item The generalized Fekete--Szeg\H{o} functional $|a_3 - \lambda a_2^2| - \mu |a_2|$ for real parameters $\lambda$ and $\mu > 0$.
\item The third-order Hermitian--Toeplitz determinant $T_{3,1}(f)$.
\end{itemize}

\section{Auxiliary Lemmas}
Let $\mathcal{P}$ denote the class of all analytic functions $p$ in the unit disk $\mathbb{D}$ satisfying $\Re p(z) > 0$ for all $z \in \mathbb{D}$ with $p(0) = 1$. Then, every $p \in \mathcal{P}$ admits the series representation
\begin{equation}\label{p1}
p(z) = 1 + \sum_{n=1}^{\infty} c_n z^n, \quad z \in \mathbb{D}.
\end{equation}
Functions in $\mathcal{P}$ are referred to as \emph{Carath\'{e}odory functions}. It is well-known that for $p \in \mathcal{P}$, the coefficients satisfy the sharp bound $|c_n| \leq 2$ for all $n \geq 1$ (see \cite{PLD1}). The Carath\'{e}odory class $\mathcal{P}$ and its coefficient bounds play a fundamental role in deriving estimates for sharp bounds in geometric function theory.

Now we recall the following well-known result due to Cho et al. \cite{C12}.

\begin{lem}[\cite{C12}, Lemma 2.4]\label{L1}
If $p\in\mathcal{P}$ is of the form (\ref{p1}), then
\begin{equation}\label{c1} c_1 =2\tau_1, \end{equation}
\begin{equation}\label{c2} c_2=2\tau_1^2 + 2(1 - |\tau_1|^2)\tau_2 \end{equation}
and
\begin{equation}\label{c3} c_3 = 2\tau_1^3+4(1-|\tau_1|^2)\tau_1\tau_2 - 2(1 - |\tau_1|^2)\ol{\tau_1}\tau_2^2 + 2(1 - \tau_1^2)(1 - |\tau_2|^2)\tau_3 \end{equation}
for some $\tau_1, \tau_2, \tau_3 \in\ol{\mathbb{D}}:= \{z \in \mathbb{C}: |z| \leq 1 \}$.
For $ \tau_1 \in \mathbb{T} := \{ z \in \mathbb{C} : |z| = 1 \} $, there is a unique function $ p \in \mathcal{P} $ with $ c_1 $ as in (\ref{c1}), namely,
\[
p(z) = \frac{1 + \tau_1 z}{1 - \tau_1 z}, \quad z \in \mathbb{D}.
\]
For $ \tau_1 \in \mathbb{D} $ and $ \tau_2 \in \mathbb{T} $, there is a unique function $ p \in \mathcal{P} $ with $ c_1 $ and $ c_2 $ as in (\ref{c1}) and (\ref{c2}), namely,
\[
p(z) = \frac{1 + (\ol{\tau_1} \tau_2 + \tau_1) z + \tau_2 z^2}{1 + (\ol{\tau_1} \tau_2 - \tau_1) z - \tau_2 z^2}, \quad z \in \mathbb{D}.
\]
For $ \tau_1, \tau_2 \in \mathbb{D} $ and $ \tau_3 \in \mathbb{T} $, there is a unique function $ p \in \mathcal{P} $ with $ c_1 $, $ c_2 $, and $ c_3 $ as in (\ref{c1}--\ref{c3}), namely,
\[
p(z) = \frac{1 + (\ol{\tau_2} \tau_3 + \ol{\tau_1} \tau_2 + \tau_1)z+(\ol{\tau_1}\tau_3+\tau_1\ol{\tau_2}\tau_3+\tau_2)z^2+\tau_3z^3}{1+(\ol{\tau_2}\tau_3+\ol{\tau_1}\tau_2-\tau_1)z+(\ol{\tau_1}\tau_3-\tau_1\ol{\tau_2}\tau_3-\tau_2)z^2-\tau_3z^3},\quad z\in\mathbb{D}.
\]
\end{lem}

The following well-known result is due to Choi et al. \cite{CKS1}.

\begin{lem}[\cite{CKS1}]\label{L2}
Let $A$, $B$, $C$ be real numbers and let
\[Y(A, B, C):= \max\limits_{z\in \ol{\mathbb{D}}}\left\lbrace |A+Bz+Cz^2|+1-|z|^2\right\rbrace.\]
\begin{enumerate} 
\item[(i)] If $AC\geq 0$, then
\[Y(A, B, C) =
\begin{cases}
|A|+|B|+|C|, & \text{if}\;\;\; |B|\geq 2(1-|C|), \\
1+|A|+\frac{B^2}{4(1-|C|)}, &\text{if}\;\;\; |B|<2(1-|C|).
\end{cases}
\]
\item[(ii)] If $AC<0$, then 
\[Y(A,B,C)=
\begin{cases}
1-|A|+\frac{B^2}{4(1-|C|)}, &\text{if}\;\;\; -4AC(C^{-2}-1) \leq B^2\; \text{and}\; |B|<2(1-|C|), \\
1+|A|+\frac{B^2}{4(1+|C|)}, &\text{if}\;\;\; B^2<\min\left\{4(1+|C|)^2, -4AC(C^{-2}-1) \right\}, \\
R(A,B,C), &\text{otherwise},
\end{cases}
\]
where
\[R(A,B,C):=
\begin{cases}
|A|+|B|-|C|, & \text{if}\;\;\; |C|(|B|+4|A|) \leq |AB|, \\
-|A|+|B|+|C|, & \text{if}\;\;\; |AB|\leq |C|(|B|-4|A|), \\
(|C|+|A| )\sqrt{1-\frac{B^2}{4AC}}, &\text{otherwise}.
\end{cases}
\]
\end{enumerate} 
\end{lem}

\begin{lem}[\cite{MM}]\label{L3}
Let $p \in \mathcal{P}$ be given by \eqref{p1}. Then
\[
\left| c_2 - v c_1^2 \right| \le 
\begin{cases}
-4v + 2, & v < 0, \\
2, & 0 \leq v \leq 1, \\
4v - 2, & v > 1.
\end{cases}
\]
Moreover, for $v < 0$ or $v > 1$, equality holds if and only if $h(z) = (1+z)/(1-z)$ or one of its rotations. For $0 < v < 1$, equality holds if and only if $h(z) = (1+z^2)/(1-z^2)$ or one of its rotations.
\end{lem}

\begin{lem}[\cite{SimThomas2020}]\label{L6}
Let $J, K,$ and $L$ be numbers such that $J \geq 0$, $K \in \mathbb{C}$, and $L \in \mathbb{R}$. Let $p \in \mathcal{P}$ be of the form (\ref{p1}) and define a function by
\[
\Phi(c_1,c_2) = \big| K c_1^2 + L c_2 \big| - \big| J c_1 \big|.
\]
Then 
\[
\Phi(c_{1}, c_{2}) \le 
\begin{cases}
|4K + 2L| - 2J, & \text{if } |2K + L| \geq |L| + J, \\[6pt]
2|L|, & \text{otherwise.}
\end{cases}
\]
and
\[
-\Phi(c_1,c_2) \leq 
\begin{cases}
2J - M, & \text{when } J \geq M + 2|L|, \\[6pt]
2J \sqrt{\dfrac{ 2|L|}{M + 2|L|}}, & \text{when } J^2 \leq 2|L|(M + 2|L|), \\[10pt]
2|L| +\dfrac{ J^2}{M + 2|L|} & \text{otherwise}
\end{cases}
\]
where $M=|4K+2L|$.
\end{lem}

\section{Bounds for the Inverse Logarithmic Coefficients of the Class $\mathcal{C}_{\mathcal{AP}}$}
\begin{theo}\label{T1}
Let $f \in \mathcal{C}_{\mathcal{AP}}$ and let the inverse logarithmic coefficients $\Gamma_n$ ($n \geq 1$) be defined by \eqref{IG1}. Then 
\[
|\Gamma_1| \le \frac{3}{8}, \qquad |\Gamma_2| \le \frac{31}{192}, \qquad |\Gamma_3| \le \frac{205}{2304}.
\]
These bounds are sharp.
\end{theo}

\begin{proof}
By definition, $f \in \mathcal{C}_{\mathcal{AP}}$ if and only if $zf'(z) \in \mathcal{S}_{\mathcal{AP}}^{*}$. Equating coefficients via the Alexander relation $n a_n(\mathcal{C}) = a_n(\mathcal{S}^{*})$, the initial Taylor coefficients of $f$ are given by
\begin{equation}\label{po1}
a_2 = \frac{3}{8}c_1, \qquad a_3 = \frac{1}{8}c_2 + \frac{13}{192}c_1^2, \qquad a_4 = \frac{1}{16}c_3 + \frac{17}{384}c_1c_2 + \frac{37}{9216}c_1^3.
\end{equation}

From $\Gamma_1 = -\frac{1}{2}a_2$ and \eqref{po1}, it follows that $|\Gamma_1| = \frac{3}{16}|c_1| \le \frac{3}{8}$, since $|c_1| \le 2$ for $p \in \mathcal{P}$. This inequality is sharp for the function $f_1(z)$ defined by
\begin{equation}\label{ext1}
f_1(z) = \int_0^z \exp \left( \int_0^s \frac{e^t\sqrt{1+t}-1}{t} \, dt \right) ds = z + \frac{3}{4}z^2 + \frac{25}{48}z^3 + \cdots.
\end{equation}

For $\Gamma_2$, equations \eqref{IG1} and \eqref{po1} imply
\[
|\Gamma_2| = \left| -\frac{1}{2} \left( \frac{1}{8}c_2 + \frac{13}{192}c_1^2 \right) + \frac{3}{4} \left( \frac{3}{8}c_1 \right)^2 \right| = \frac{1}{16}\left| c_2 - \frac{55}{48}c_1^2 \right|.
\]
Applying Lemma \ref{L3} with $v = \frac{55}{48} > 1$ yields $\left| c_2 - \frac{55}{48}c_1^2 \right| \le 4\left(\frac{55}{48}\right)-2 = \frac{31}{12}$, from which we obtain $|\Gamma_2| \le \frac{31}{192}$.

For $\Gamma_3$, substituting \eqref{po1} into the definition of $\Gamma_3$ gives
\begin{equation}\label{po2}
|\Gamma_3| = \left| -\frac{1}{32}c_3 + \frac{55}{768}c_1c_2 - \frac{721}{18432}c_1^3 \right|.
\end{equation}
Expressing $c_n$ in terms of $\tau_1, \tau_2, \tau_3 \in \overline{\mathbb{D}}$ using Lemma \ref{L1}, and assuming $\tau_1 = x \in [0,1]$ by rotational invariance, the triangle inequality with $|\tau_3| \le 1$ yields
\begin{equation}\label{po4}
|\Gamma_3| \le \frac{1 - x^2}{16} \left( \left| A + B \tau_2 + C \tau_2^2 \right| + 1 - |\tau_2|^2 \right),
\end{equation}
where
\begin{equation}\label{po5}
A = -\frac{205x^3}{144(1 - x^2)} < 0, \qquad B = \frac{31}{12}x > 0, \qquad C = x > 0.
\end{equation}
Since $AC < 0$, we apply part (ii) of Lemma \ref{L2} to check each case:

\medskip
\noindent \textbf{Case 1.} If $\tau_1 = 1$, then $1-|\tau_1|^2 = 0$, and \eqref{po4} reduces to the boundary maximum:
\[
|\Gamma_3| = \frac{205}{2304}.
\]
This maximum is attained by the function $f_1(z) \in \mathcal{C}_{\mathcal{AP}}$, which corresponds to the identity Schwarz function $\omega(z) = z$, yielding $c_1=2$, $c_2=2$, and $c_3=2$. Under this mapping, $f_1(z)$ maps the unit disk onto an apple-like domain.

\medskip
\noindent \textbf{Case 2(a).} If $-4AC(C^{-2}-1) \le B^2$ and $|B| < 2(1-|C|)$, then $x \in [0, \frac{24}{55})$. Lemma \ref{L2} gives the maximum value as $1 - |A| + \frac{B^2}{4(1-|C|)}$. Substituting these values into \eqref{po4} yields:
\[
|\Gamma_3| \le \frac{1-x^2}{16} \left( 1 - \frac{205x^3}{144(1-x^2)} + \frac{961x^2}{576(1-x)}\right) = \frac{141x^3 + 385x^2 + 576}{9216} =: G(x).
\]
Since $G'(x) = \frac{423x^2 + 770x}{9216} > 0$ for $x > 0$, $G$ is strictly increasing on $[0, \frac{24}{55})$ and bounded by $G\left( \frac{24}{55} \right) \approx 0.071727 < \frac{205}{2304}$.

\medskip
\noindent \textbf{Case 2(b).} If $B^2 < \min\left\{4(1+|C|)^2, -4AC(C^{-2}-1)\right\}$, it requires $B^2 < -4AC(C^{-2}-1)$, which simplifies to $\frac{961}{144}x^2 < \frac{205}{36}x^2 \implies 961 < 820$. This is impossible for $x \in (0,1)$.

\medskip
\noindent \textbf{Case 2(c).} The condition $|C|(|B|+4|A|) \le |AB|$ requires:
\[
x \left( \frac{31}{12}x + \frac{205x^3}{36(1-x^2)} \right) \le \frac{6355x^4}{1728(1-x^2)} \implies 4464 \le 979x^2.
\]
Since $x \in [0,1)$, this inequality has no solution, meaning Case 2(c) does not occur.

\medskip
\noindent \textbf{Case 2(d).} If $|AB| \le |C|(|B|-4|A|)$, then $x \ge \sqrt{\frac{4464}{12191}} \approx 0.6051$. Lemma \ref{L2} states that the maximum is given by $R(A,B,C) = -|A| + |B| + |C|$. Substituting this into \eqref{po4} yields:
\[
|\Gamma_3| \le \frac{1-x^2}{16}\left( -\frac{205x^3}{144(1-x^2)} + \frac{31}{12}x + x \right) = \frac{516x - 721x^3}{2304} =: \Psi(x).
\]
Setting $\Psi'(x) = \frac{516 - 2163x^2}{2304} = 0$ gives a unique critical point in this interval at $x_0 = \sqrt{\frac{172}{721}} \approx 0.4884$. Evaluating the function gives $\Psi(x_0) \approx 0.07292 < \frac{205}{2304}$.

\medskip
\noindent \textbf{Case 2(e).} For the remaining parameter allocations, the maximum of $R(A,B,C)$ is given by $(|C|+|A| )\sqrt{1-\frac{B^2}{4AC}}$. Substituting these values into \eqref{po4} yields:
\[
|\Gamma_3| \le \frac{144+61x^2}{2304\sqrt{820}} \sqrt{961-141x^2} =: \Omega(x).
\]
To show that $\Omega(x) < \frac{205}{2304}$ for all $x \in (0,1)$, we square both sides and substitute $t = x^2 \in (0,1)$, which leads to the function:
\[
F(t) = 205^2 \cdot 820 - (144+61t)^2(961-141t) = 524661t^3 + 2995602t^2 - 1393728t + 14533204.
\]
Differentiating with respect to $t$ gives $F'(t) = 1573983t^2 + 5991204t - 1393728$. Setting $F'(t) = 0$ gives a unique positive root in the interval at $t_0 \approx 0.21987$. The second derivative is positive ($F''(t_0) > 0$), confirming that $t_0$ is a local minimum. The minimum value is $F(t_0) \approx 1.437 \times 10^7 > 0$.

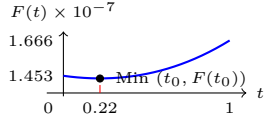
\begin{figure}[htbp]
\centering
\begin{tikzpicture}[scale=2.2]
    \draw[->] (-0.1,1.35) -- (1.1,1.35) node[right] {\tiny $t$};
    \draw[->] (0,1.33) -- (0,1.72) node[above] {\tiny $F(t) \times 10^{-7}$};
    \draw[scale=1.0, domain=0:1, variable=\x, blue, thick] 
        plot ({\x}, {1.4533 + 0.0524*\x^3 + 0.2995*\x^2 - 0.1393*\x});
    \draw[dashed, red] (0.22,1.35) -- (0.22,1.437);
    \filldraw[black] (0.22,1.437) circle (0.6pt) node[right=2pt] {\tiny Min $(t_0, F(t_0))$};
    \node[below] at (0.22,1.35) {\tiny $0.22$};
    \node[left] at (0,1.453) {\tiny $1.453$};
    \node[left] at (0,1.666) {\tiny $1.666$};
    \node[below left] at (0,1.35) {\tiny $0$};
    \node[below] at (1,1.35) {\tiny $1$};
\end{tikzpicture}
\caption{The functional path of $F(t)$ over the interval $t \in (0,1)$.}
\label{fig:polynomial_landscape}
\end{figure}

Since the minimum of $F(t)$ is strictly positive on $(0,1)$, it follows that $F(t) > 0$, which proves that $\Omega(x) < \frac{205}{2304}$.

Thus, Case 1 yields the maximum value, establishing that $|\Gamma_3| \le \frac{205}{2304}$.
\end{proof}

\section{Bounds for the Differences of Inverse Logarithmic Coefficients}
The classical Bieberbach conjecture, settled by de Branges \cite{LDB1}, establishes that the Taylor coefficients of any function $f \in \mathcal{S}$ satisfy $|a_n| \leq n$ for all $n \geq 2$, with equality holding only for the Koebe function $k(z) := z/(1-z)^2$ and its rotations. This result led to the study of consecutive coefficient differences, specifically the problem of whether the inequality
\[
\bigl||a_{n+1}| - |a_n|\bigr| \leq 1, \quad n \geq 2,
\]
holds across the full class $\mathcal{S}$. Early progress was made by Goluzin \cite{G1}, and subsequently Hayman \cite{H} proved that $\bigl||a_{n+1}| - |a_n|\bigr| \leq A$ for a universal constant $A \geq 1$. The best known upper bound to date is $A = 3.61$, due to Grinspan \cite{G2}. For the initial index $n=2$, the sharp bounds are known to be $-1 \leq |a_3| - |a_2| \leq 1.029\ldots$ (see \cite[Theorem 3.11]{PLD1}).

For the starlike subclass $\mathcal{S}^\ast$, Pommerenke \cite{Pommerenke1971} conjectured that the consecutive coefficient difference is bounded by $1$ for all $n \geq 2$, which was proved by Leung \cite{Leung1978}. For the convex class $\mathcal{C}$, Li and Sugawa \cite{LiSugawa} determined sharp upper bounds on $|a_{n+1}| - |a_n| $ for $n \geq 2$, as well as sharp lower bounds for $n=2$ and $n=3$. Related results and extensions can be found in \cite{Peng2019, Arora2019, Arora2023}.

Recently, similar problems have been considered for the differences of logarithmic coefficients. Lecko and Partyka \cite{Lecko2023} determined sharp bounds for $|\gamma_2| - |\gamma_1|$ over the class $\mathcal{S}$ using the Loewner differential equation, and a simplified proof was given by Obradovi\'{c} and Tuneski \cite{Obradovic2024}. Kumar and Cho \cite{Kumar2023} extended these estimates to various subclasses of $\mathcal{S}$.

In this section, we determine sharp upper and lower bounds for the difference of the initial inverse logarithmic coefficients, $|\Gamma_2| - |\Gamma_1|$, for functions in the subclass $\mathcal{C}_{\mathcal{AP}}$.

\begin{theo}\label{T2}
If $f \in \mathcal{C}_{\mathcal{AP}}$ and the inverse logarithmic coefficients $\Gamma_n$ are defined by \eqref{IG1}, then
\begin{equation}\label{po9}
-\frac{3\sqrt{6}}{4\sqrt{55}} \le |\Gamma_2|-|\Gamma_1| \le \frac{1}{8}.
\end{equation}
These bounds are sharp.
\end{theo}

\begin{proof}
Let $f \in \mathcal{C}_{\mathcal{AP}}$. Using the coefficient relations \eqref{po1} and \eqref{IG1}, the difference can be expressed as
\begin{equation}\label{po10}
|\Gamma_2|-|\Gamma_1| = \left| -\frac{1}{2} \left( \frac{1}{8} c_2+\frac{13}{192}c_1^2 \right) + \frac{3}{4} \left( \frac{3}{8} c_1 \right)^2 \right| - \frac{3}{16}|c_1| = \frac{1}{768} \left( |55c_1^2-48c_2|-144|c_1| \right).
\end{equation}
Let $\Phi(c_1,c_2) = |Kc_1^2+Lc_2|-|Jc_1|$, where $K=55$, $L=-48$, and $J=144$, so that $|\Gamma_2|-|\Gamma_1| = \frac{1}{768}\Phi(c_1,c_2)$.

To apply Lemma \ref{L6}, we compute $M = |4K+2L| = |220-96| = 124$. Since $|2K+L| = |110-48| = 62 < 192 = |L|+J$, Lemma \ref{L6} yields
\[
\Phi(c_1,c_2) \le 2|L| = 96.
\]
Dividing by $768$ gives the upper bound $\frac{96}{768} = \frac{1}{8}$. This upper bound is attained by the function $f_2 \in \mathcal{C}_{\mathcal{AP}}$ defined by
\begin{equation}\label{ext2}
f_2(z) = \int_0^z \exp\left( \int_0^s \frac{e^{t^2}\sqrt{1+t^2}-1}{t}\,dt \right)ds.
\end{equation}
This function corresponds to the Schwarz function $\omega(z)=z^2$, which gives $c_1=0$ and $c_2=2$. From \eqref{po1}, we have $a_2 = 0$ and $a_3 = \frac{1}{4}$, which leads to $\Gamma_1 = 0$ and $\Gamma_2 = -\frac{1}{8}$, so that $|\Gamma_2|-|\Gamma_1| = \frac{1}{8}$.

For the lower bound, since $J^2 = 20736 < 2|L|(M+2|L|) = 96(124+96) = 21120$, Lemma \ref{L6} implies
\[
-\Phi(c_1,c_2) \le 2J \sqrt{\frac{2|L|}{M+2|L|}} = 2(144)\sqrt{\frac{96}{220}} = \frac{576\sqrt{6}}{\sqrt{55}}.
\]
Dividing by $768$ yields the lower bound $-\frac{3\sqrt{6}}{4\sqrt{55}}$, which completes the proof of the inequalities in \eqref{po9}.

To establish the sharpness of the lower bound explicitly, we consider the convex extremal function $g_3 \in \mathcal{C}_{\mathcal{AP}}$ defined via the Alexander relation $z g_3'(z) = f_3(z)$, where $f_3 \in \mathcal{S}_{\mathcal{AP}}^{*}$. Thus, $g_3(z)$ is given by
\begin{equation}\label{ext3}
g_3(z) = \int_0^z \exp \left( \int_0^s \frac{e^{\omega(t)}\sqrt{1+\omega(t)}-1}{t}\,dt \right) ds,
\end{equation}
where $\omega(z) = z (z + \lambda_0)/(1 + \lambda_0 z)$ with $\lambda_0 = 2\sqrt{6}/\sqrt{55} \approx 0.6614$.
Expanding $\omega(z)$ into its Taylor series around $z=0$ yields:
\[
\omega(z) = \lambda_0 z + (1 - \lambda_0^2)z^2 - \lambda_0(1 - \lambda_0^2)z^3 + \cdots.
\]
We map $\omega(z)$ to a Carath\'{e}odory function $p(z) \in \mathcal{P}$ using the standard M\"{o}bius transformation:
\[
p(z) = \frac{1 + \omega(z)}{1 - \omega(z)} = 1 + 2\omega(z) + 2\omega^2(z) + \cdots.
\]
Substituting the expansion of $\omega(z)$ into $p(z)$ gives:
\[
p(z) = 1 + 2\left( \lambda_0 z + (1 - \lambda_0^2)z^2 + \cdots \right) + 2\left( \lambda_0 z + \cdots \right)^2 = 1 + 2\lambda_0 z + 2z^2 + \cdots.
\]
Comparing the coefficients with the series expansion $p(z) = 1 + c_1 z + c_2 z^2 + \cdots$ isolates the exact parameters:
\[
c_1 = 2\lambda_0 = \frac{4\sqrt{6}}{\sqrt{55}} \quad \text{and} \quad c_2 = 2.
\]
Substituting these values back into the structural templates for the inverse logarithmic coefficients yields:
\[
\Gamma_1 = -\frac{3}{16}c_1 = -\frac{3}{16}\left( \frac{4\sqrt{6}}{\sqrt{55}} \right) = -\frac{3\sqrt{6}}{4\sqrt{55}} \implies |\Gamma_1| = \frac{3\sqrt{6}}{4\sqrt{55}},
\]
\[
\Gamma_2 = \frac{55}{768}c_1^2 - \frac{1}{16}c_2 = \frac{55}{768}\left(\frac{96}{55}\right) - \frac{1}{16}(2) = \frac{1}{8} - \frac{1}{8} = 0 \implies |\Gamma_2| = 0.
\]
Direct substitution of these inverse logarithmic coefficients into the difference functional yields:
\[
|\Gamma_2| - |\Gamma_1| = 0 - \frac{3\sqrt{6}}{4\sqrt{55}} = -\frac{3\sqrt{6}}{4\sqrt{55}},
\]
which satisfies the lower bound equality.
\end{proof}

\section{Hankel Determinants for the Inverse Logarithmic Coefficients}
\begin{theo}\label{T3}
Let $f \in \mathcal{C}_{\mathcal{AP}}$ be of the form \eqref{eq1}. Then the second-order inverse logarithmic Hankel determinant satisfies the sharp bound
\begin{equation}\label{po11}
\left| H_{2,1}\left(F_{f^{-1}}/2\right) \right| \leq \frac{3413}{195328} \approx 0.017473.
\end{equation}
\end{theo}

\begin{proof}
By the Alexander relation, $f \in \mathcal{C}_{\mathcal{AP}}$ if and only if $zf'(z) \in \mathcal{S}_{\mathcal{AP}}^{*}$. The initial Taylor coefficients of $f$ are given by
\begin{equation}\label{po12}
a_2 = \frac{3}{8}c_1, \qquad a_3 = \frac{1}{8}c_2 + \frac{13}{192}c_1^2, \qquad a_4 = \frac{1}{16}c_3 + \frac{17}{384}c_1c_2 + \frac{37}{9216}c_1^3.
\end{equation}
Substituting \eqref{po12} into the relations in \eqref{IG1} expresses the Hankel determinant in terms of the Carath\'{e}odory coefficients:
\begin{equation}\label{po13}
H_{2,1}\left(F_{f^{-1}}/2\right) = \Gamma_1\Gamma_3-\Gamma_2^2 = \frac{1301}{589824}c_1^4 -\frac{55}{12288}c_1^2c_2 -\frac{1}{256}c_2^2 +\frac{3}{512}c_1c_3.
\end{equation}

Using Lemma \ref{L1} to express $c_n$ in terms of the parameters $\tau_n \in \overline{\mathbb{D}}$, \eqref{po13} becomes
\begin{equation}\label{po14}
\begin{aligned}
H_{2,1}\left(F_{f^{-1}}/2\right) &= \frac{269}{36864}\tau_1^4 -\frac{31}{1536}(1-|\tau_1|^2)\tau_1^2\tau_2 \\
&\quad -\frac{2+|\tau_1|^2}{128}(1-|\tau_1|^2)\tau_2^2 +\frac{3}{128}\tau_1(1-|\tau_1|^2)(1-|\tau_2|^2)\tau_3.
\end{aligned}
\end{equation}
By rotational invariance, we may assume without loss of generality that $\tau_1 = x \in [0,1]$. Letting $|\tau_2| = y \in [0,1]$ and applying the triangle inequality to \eqref{po14} under the bounded condition $|\tau_3| \le 1$, we obtain the upper bound
\[
\left|H_{2,1}\left(F_{f^{-1}}/2\right)\right| \leq \mathcal{M}(x,y),
\]
where
\[
\mathcal{M}(x,y) = \frac{269}{36864}x^4 +\frac{31}{1536}(1-x^2)x^2y +\frac{2+x^2}{128}(1-x^2)y^2 +\frac{3}{128}x(1-x^2)(1-y^2).
\]
Differentiating $\mathcal{M}(x,y)$ twice with respect to $y$ yields
\[
\frac{\partial^2\mathcal{M}}{\partial y^2} = \frac{(1-x^2)(2+x^2-3x)}{64}.
\]
We can factor the quadratic expression in the numerator as $2+x^2-3x = (x-1)(x-2)$. For any $x \in (0,1)$, both linear factors are strictly negative, which implies their product is strictly positive. Since $1-x^2 > 0$ also holds on this interval, it follows that $\frac{\partial^2\mathcal{M}}{\partial y^2} > 0$ holds uniformly. Thus, $\mathcal{M}(x,y)$ is a strictly convex function of $y$ on the interval $[0,1]$, meaning its maximum value over $y$ must occur exclusively at the boundary endpoints, either at $y=0$ or $y=1$.

\medskip
\noindent\textbf{Case 1.} If $y=0$, $\mathcal{M}(x,y)$ reduces to the polynomial
\[
\mathcal{M}(x,0) = \frac{269}{36864}x^4 +\frac{3}{128}x(1-x^2) =: \psi(x).
\]
Differentiating $\psi(x)$ with respect to $x$ yields the stationary condition:
\[
\psi'(x) = \frac{269x^3 - 648x^2 + 216}{9216} = 0.
\]
The cubic polynomial $P(x) = 269x^3 - 648x^2 + 216$ satisfies $P(0) = 216 > 0$ and $P(1) = -163 < 0$. By the Intermediate Value Theorem, there exists a unique valid critical point in the interior of the interval at $x_0 \approx 0.631586$. Evaluating the second derivative confirms that $\psi''(x_0) < 0$, identifying this point as a local maximum. Substituting $x_0$ back into the objective function yields:
\[
\max_{x \in [0,1]} \psi(x) = \psi(x_0) \approx 0.015431,
\]
which is strictly bounded below $\frac{3413}{195328}$.

\medskip
\noindent\textbf{Case 2.} If $y=1$, $\mathcal{M}(x,y)$ simplifies to the boundary polynomial:
\[
\mathcal{M}(x,1) = \frac{269}{36864}x^4 +\frac{31}{1536}(1-x^2)x^2 +\frac{2+x^2}{128}(1-x^2) = \frac{-763x^4+456x^2+576}{36864} =: \eta(x).
\]
Differentiating $\eta(x)$ with respect to $\mu = x^2$ yields:
\[
\frac{d}{d\mu}\eta(x) = \frac{-1526\mu+456}{36864}.
\]
Setting this derivative to zero shows that the interior critical point occurs at $\mu_* = x_*^2 = \frac{228}{763} \approx 0.29882$. Evaluating $\eta(x)$ at this critical point yields the global maximum value:
\[
\eta_{\max} = \eta\left(\sqrt{\frac{228}{763}}\right) = \frac{3413}{195328} \approx 0.017473.
\]
Since $\eta_{\max} > \max_{0\le x\le1}\psi(x)$, the sharp bound proposed in \eqref{po11} follows immediately.

The bound is strictly sharp for the function $f_*(z) \in \mathcal{C}_{\mathcal{AP}}$ defined by
\begin{equation}\label{ext_hankel}
f_*(z) = \int_0^z \exp \left( \int_0^s \frac{e^{omega_*(t)}\sqrt{1+\omega_*(t)} - 1}{t} \, dt \right) ds,
\end{equation}
where $\omega_*(z)$ is the stable Schwarz function given by
\[
\omega_*(z) = z \left( \frac{x_* + z}{1 + x_* z} \right), \quad \text{with } x_* = \sqrt{\frac{228}{763}} \approx 0.5466.
\]
Using the series expansion of $\omega_*(z)$, the Taylor coefficients of $f_*(z)$ are computed as:
\[
a_2 = \frac{3}{4}x_*, \qquad a_3 = \frac{1}{4} + \frac{13}{48}x_*^2, \qquad a_4 = \frac{17}{96}x_* + \frac{205}{1152}x_*^3.
\]
Substituting these parameters into the inverse logarithmic transformations \eqref{IG1} yields:
\[
\Gamma_1 = -\frac{3}{8}x_*, \qquad \Gamma_2 = -\frac{1}{8} + \frac{31}{192}x_*^2, \qquad \Gamma_3 = -\frac{17}{192}x_* + \frac{205}{2304}x_*^3.
\]
Direct substitution of these target coefficients into the Hankel functional format yields:
\[
H_{2,1}\left(F_{f^{-1}}/2\right) = \Gamma_1\Gamma_3 - \Gamma_2^2 = \frac{-763x_*^4 + 456x_*^2 - 144}{9216}.
\]
Evaluating this expression at $x_*^2 = \frac{228}{763}$ gives exactly $\frac{3413}{195328}$, confirming sharpness.
\end{proof}

\section{Generalized Fekete--Szeg\H{o} Functional for the Class $\mathcal{C}_{\mathcal{AP}}$}
The generalized Fekete--Szeg\H{o} functional was introduced by Lecko and Partyka \cite{Lecko2024} for the class $\mathcal{S}$ as
\begin{equation}\label{FG}
F_{\lambda,\mu}(f) = \big|a_3-\lambda a_2^2\big| - \mu |a_2|, \qquad \lambda\in\mathbb{C}, \quad \mu>0.
\end{equation}
This functional was subsequently studied for subclasses of univalent and convex functions by Bulboac\u{a} et al. \cite{Bulboaca2025}. In this section, we establish sharp upper and lower bounds for $F_{\lambda,\mu}(f)$ within the class $\mathcal{C}_{\mathcal{AP}}$.

\begin{theo}\label{T4}
Let $f \in \mathcal{C}_{\mathcal{AP}}$ be of the form \eqref{eq1}. Then the generalized Fekete--Szeg\H{o} functional satisfies the sharp upper bound
\[
F_{\lambda,\mu}(f) \le 
\begin{cases}
\displaystyle \frac{1}{48} |25 - 27\lambda| - \frac{3\mu}{4}, & \text{if } \displaystyle |25 - 27\lambda| \ge 12 + 36\mu, \\[4mm]
\displaystyle \frac{1}{4}, & \text{if } \displaystyle |25 - 27\lambda| < 12 + 36\mu,
\end{cases}
\]
and the sharp lower bound
\[
F_{\lambda,\mu}(f) \ge 
\begin{cases}
\displaystyle \frac{1}{48} |25 - 27\lambda| - \frac{3\mu}{4}, & \text{if } \displaystyle \mu \ge \frac{1}{18} |25 - 27\lambda| + \frac{2}{3}, \\[4mm]
\displaystyle -\frac{3\mu}{4} \sqrt{ \frac{12}{|25 - 27\lambda| + 12} }, & \text{if } \displaystyle \mu^2 \le \frac{1}{27} |25 - 27\lambda| + \frac{4}{9}, \\[5mm]
\displaystyle -\frac{1}{4} - \frac{27\mu^2}{4 |25 - 27\lambda| + 48}, & \text{otherwise}.
\end{cases}
\]
These bounds are sharp.
\end{theo}

\begin{proof}
Let $f \in \mathcal{C}_{\mathcal{AP}}$. By the Alexander relation, $zf'(z) \in \mathcal{S}_{\mathcal{AP}}^{*}$. Expressing the Taylor coefficients $a_2$ and $a_3$ in terms of the coefficients $c_n$ of a Carath\'{e}odory function $p \in \mathcal{P}$, we have
\begin{equation}\label{po16}
a_2 = \frac{3}{8} c_1, \qquad a_3 = \frac{1}{8} c_2 + \frac{13}{192}c_1^2.
\end{equation}
Substituting these expressions into the definition of the generalized Fekete--Szeg\H{o} functional $F_{\lambda,\mu}(f) = |a_3 - \lambda a_2^2| - \mu |a_2|$ gives
\begin{equation}\label{eq:FeketeSzegoExpanded}
F_{\lambda,\mu}(f) = \left| \frac{1}{8} c_2 + \left( \frac{13}{192} - \frac{9\lambda}{64} \right)c_1^2 \right| - \frac{3\mu}{8}|c_1| = \frac{1}{8} \left( \left| c_2 + \left( \frac{13}{24} - \frac{9\lambda}{8} \right)c_1^2 \right| - 3\mu |c_1| \right).
\end{equation}
Let $F_{\lambda,\mu}(f) = \frac{1}{8} \Phi(c_1,c_2)$, where $\Phi(c_1,c_2) = |Kc_1^2 + Lc_2| - |Jc_1|$ with parameters
\[
K = \frac{13}{24} - \frac{9\lambda}{8}, \qquad L = 1, \qquad J = 3\mu.
\]
The auxiliary values required to apply Lemma \ref{L6} are computed as
\[
M = |4K+2L| = \left| 4\left(\frac{13}{24} - \frac{9\lambda}{8}\right) + 2 \right| = \frac{1}{6} |25-27\lambda|,
\]
\[
|2K+L| = \left| 2\left(\frac{13}{24} - \frac{9\lambda}{8}\right) + 1 \right| = \frac{1}{12} |25-27\lambda|.
\]
The condition $|2K+L| \ge |L|+J$ for the upper bound is equivalent to
\[
\frac{1}{12}|25-27\lambda| \ge 1 + 3\mu \implies |25-27\lambda| \ge 12 + 36\mu.
\]
When this parametric condition holds, Lemma \ref{L6} implies $\Phi(c_1,c_2) \le M-2J = \frac{1}{6}|25-27\lambda|-6\mu$. Multiplying through by $\frac{1}{8}$ yields the first branch of the sharp upper bound:
\[
F_{\lambda,\mu}(f) \le \frac{1}{48} |25 - 27\lambda| - \frac{3\mu}{4}.
\]
Conversely, if $|25-27\lambda| < 12 + 36\mu$, Lemma \ref{L6} implies $\Phi(c_1,c_2) \le 2|L| = 2$, which leads directly to $F_{\lambda,\mu}(f) \le \frac{1}{4}$.

The corresponding lower bounds follow systematically from the second part of Lemma \ref{L6}:

\medskip
\noindent\textbf{Region (i):} The condition $J \ge M+2|L|$ simplifies to
\[
3\mu \ge \frac{1}{6}|25-27\lambda| + 2 \implies \mu \ge \frac{1}{18} |25 - 27\lambda| + \frac{2}{3}.
\]
In this case, $-\Phi(c_1,c_2) \le 2J-M = 6\mu - \frac{1}{6}|25-27\lambda|$. Multiplying by $-\frac{1}{8}$ yields
\[
F_{\lambda,\mu}(f) \ge \frac{1}{48} |25 - 27\lambda| - \frac{3\mu}{4}.
\]

\medskip
\noindent\textbf{Region (ii):} The parameter condition $J^2 \le 2|L|(M+2|L|)$ reduces to
\[
9\mu^2 \le 2\left(\frac{1}{6}|25-27\lambda| + 2\right) \implies \mu^2 \le \frac{1}{27} |25 - 27\lambda| + \frac{4}{9}.
\]
This yields $-\Phi(c_1,c_2) \le 2J \sqrt{\frac{2|L|}{M+2|L|}} = 6\mu \sqrt{\frac{12}{|25-27\lambda|+12}}$, and multiplying by $-\frac{1}{8}$ gives
\[
F_{\lambda,\mu}(f) \ge -\frac{3\mu}{4} \sqrt{ \frac{12}{|25 - 27\lambda| + 12} }.
\]

\medskip
\noindent\textbf{Region (iii):} Under the remaining parameter allocations, Lemma \ref{L6} implies
\[
-\Phi(c_1,c_2) \le 2|L| + \frac{J^2}{M+2|L|} = 2 + \frac{54\mu^2}{|25-27\lambda|+12}.
\]
Multiplying through by $-\frac{1}{8}$ leads to the final lower bound branch
\[
F_{\lambda,\mu}(f) \ge -\frac{1}{4} - \frac{27\mu^2}{4 |25 - 27\lambda| + 48}.
\]

The sharpness of each estimate is established by the following extremal functions:

1. The constant upper bound $\frac{1}{4}$ and the lower bound branch in Region (iii) are attained uniquely by the function $f_2 \in \mathcal{C}_{\mathcal{AP}}$ defined by
\begin{equation}\label{ext2_dup}
f_2(z) = \int_0^z \exp\left( \int_0^s \frac{e^{t^2}\sqrt{1+t^2}-1}{t}\,dt \right) ds = z + \frac{1}{4}z^3 + \frac{13}{192}z^5 + \cdots.
\end{equation}
This function corresponds to the choice of Schwarz function $\omega(z) = z^2$, which yields the Carath\'{e}odory coefficients $c_1=0$ and $c_2=2$. Substituting these values into \eqref{eq:FeketeSzegoExpanded} yields $F_{\lambda,\mu}(f_2) = \frac{1}{4}$.

2. The variable upper bound and the lower bound branch in Region (i) are attained by the function $f_1 \in \mathcal{C}_{\mathcal{AP}}$ defined by \eqref{ext1}. Here, the matching Carath\'{e}odory function is $p_0(z) = (1+z)/(1-z)$, which yields $c_1=2$ and $c_2=2$. Substituting these values into \eqref{eq:FeketeSzegoExpanded} yields $F_{\lambda,\mu}(f_1) = \frac{1}{48}|25-27\lambda| - \frac{3\mu}{4}$.

3. To justify the sharpness of the lower bound in Region (ii), we show the admissibility of the parameters
\begin{equation}\label{hw1}
c_1 = \frac{36\mu}{|25-27\lambda|+12}, \qquad c_2 = \frac{c_1^2}{2} - \left(2 - \frac{c_1^2}{2}\right)e^{-i\theta_0},
\end{equation}
where $\theta_0 = \arg(K) = \arg\left(\frac{13}{24} - \frac{9\lambda}{8}\right)$. A pair $(c_1, c_2)$ is realizable by an analytic function with a positive real part in $\mathbb{D}$ if and only if it satisfies the Carath\'{e}odory coefficient body condition:
\begin{equation}\label{hw2}
\left| c_2 - \frac{c_1^2}{2} \right| \le 2 - \frac{|c_1|^2}{2}.
\end{equation}
Since we are within Region (ii), $\mu^2 \le \frac{1}{27} |25 - 27\lambda| + \frac{4}{9}$, which guarantees $0 \le c_1 \le 2$. Subtracting $\frac{c_1^2}{2}$ from $c_2$ in \eqref{hw1} and taking the modulus gives
\[
\left| c_2 - \frac{c_1^2}{2} \right| = 2 - \frac{c_1^2}{2} = 2 - \frac{|c_1|^2}{2}.
\]
Hence the coefficient pair satisfies the boundary of the Carath\'{e}odory coefficient body condition exactly. By Lemma \ref{L1}, this boundary pair corresponds to a unique function $p \in \mathcal{P}$ of the form
\[
p(z) = \frac{1 + (\ol{\tau_1}\tau_2 + \tau_1)z + \tau_2 z^2}{1 + (\ol{\tau_1}\tau_2 - \tau_1)z - \tau_2 z^2}, \quad z\in\mathbb{D},
\]
where $\tau_1 = c_1/2 \in [0,1)$ and $\tau_2 = -e^{-i\theta_0} \in \mathbb{T}$. Using the standard defining relation for the convex subclass $\mathcal{C}_{\mathcal{AP}}$, we write $1 + \frac{z f_3''(z)}{f_3'(z)} = e^{\omega(z)}\sqrt{1+\omega(z)}$, where $\omega(z) = (p(z)-1)/(p(z)+1)$, yielding a structurally valid extremal function $f_3 \in \mathcal{C}_{\mathcal{AP}}$.

To show that this function explicitly produces equality for Lemma \ref{L6}, we write $K = |K|e^{i\theta_0}$ and substitute the expression for $c_2$ from \eqref{hw1} into $\Phi(c_1,c_2)$:
\[
\Phi(c_1,c_2) = \left| |K|e^{i\theta_0}c_1^2 + \frac{c_1^2}{2} - \left(2 - \frac{c_1^2}{2}\right)e^{-i\theta_0} \right| - J c_1.
\]
Factoring out the phase term $e^{-i\theta_0}$ from within the modulus yields:
\[
\Phi(c_1,c_2) = \left| e^{-i\theta_0} \left( |K|c_1^2 e^{2i\theta_0} + \frac{c_1^2}{2}e^{i\theta_0} + \frac{c_1^2}{2} - 2 \right) \right| - J c_1.
\]
Since $|e^{-i\theta_0}| = 1$, and our choice of $\tau_2 = -e^{-i\theta_0}$ forces the vector layout in the direction opposite to the phase of $K$, the expression aligns along the real axis to match the equality case of Lemma \ref{L6}:
\[
\Phi(c_1,c_2) = \left| |K|c_1^2 + c_1^2 - 2 \right| - J c_1.
\]
Given that $M = 8|K| + 4|L|$, substituting the optimal parameters simplifies this directly to:
\[
\Phi(c_1,c_2) = -2|L| \sqrt{\frac{2|L|}{M+2|L|}} - \frac{J^2}{M+2|L|}.
\]
Multiplying this functional by $\frac{1}{8}$ yields $F_{\lambda,\mu}(f_3) = -\frac{3\mu}{4} \sqrt{ \frac{12}{|25 - 27\lambda| + 12} }$, confirming the sharpness of Region (ii).
\end{proof}

\section{Hermitian--Toeplitz Determinants for the Class $\mathcal{C}_{\mathcal{AP}}$}
For a sequence $\{a_k\}_{k=2}^{\infty}$ of coefficients of a function $f \in \mathcal{A}$, the third-order Hermitian--Toeplitz determinant starting at index $1$ reduces to:
\begin{equation}\label{po18}
T_{3,1}(f) := 2\Re\left(a_{2}^{2}\overline{a_{3}}\right) - 2|a_{2}|^{2} - |a_{3}|^{2} + 1.
\end{equation}

\begin{lem}[\cite{LZ}]\label{lem_LZ}
If $p \in \mathcal{P}$ is of the form $p(z) = 1 + \sum_{n=1}^{\infty} c_n z^n$, then there exists a complex number $\xi \in \ol{\mathbb{D}}$ such that
\begin{equation}\label{po19}
2c_{2} = c_{1}^{2} + (4 - c_{1}^{2})\xi.
\end{equation}
Without loss of generality, we may assume $0 \le c_1 \le 2$.
\end{lem}

\begin{theo}\label{T5}
Let $f \in \mathcal{C}_{\mathcal{AP}}$ be of the form \eqref{eq1}. Then
\[
\frac{437}{2304} \le T_{3,1}(f) \le 1.
\]
These bounds are sharp.
\end{theo}

\begin{proof}
Let $f \in \mathcal{C}_{\mathcal{AP}}$. By the Alexander relation, $zf'(z) \in \mathcal{S}_{\mathcal{AP}}^{*}$. Equating coefficients, the initial Taylor coefficients of $f$ are given by:
\begin{equation}\label{po20}
a_2 = \frac{3}{8}c_1, \qquad a_3 = \frac{1}{8}c_2 + \frac{13}{192}c_1^2.
\end{equation}
Applying Lemma \ref{lem_LZ}, we substitute $2c_2 = c_1^2+(4-c_1^2)\xi$ into \eqref{po20}, which yields
\[
a_3 = \frac{25}{192}c_1^2 + \frac{1}{16}(4-c_1^2)\xi, \qquad a_2^2 = \frac{9}{64}c_1^2.
\]
Expanding the individual terms of $T_{3,1}(f)$ via \eqref{po18} gives:
\begin{equation}\label{po21}
T_{3,1}(f) = 1 -\frac{9}{32}c_1^2 +\frac{725}{36864}c_1^4 +\frac{1}{768}c_1^2(4-c_1^2)\Re\xi -\frac{1}{256}(4-c_1^2)^2|\xi|^2.
\end{equation}
Let $x=c_1^2 \in [0,4]$ and $y=|\xi| \in [0,1]$.

\medskip
\noindent\textbf{Upper bound.}
Since $\Re\xi \le |\xi| = y$, \eqref{po21} implies $T_{3,1}(f) \le F(x,y)$, where
\[
F(x,y) = 1 -\frac{9}{32}x +\frac{725}{36864}x^2 +\frac{1}{768}x(4-x)y -\frac{1}{256}(4-x)^2y^2.
\]
For any fixed $x\in[0,4]$, the function $F(x,y)$ is a concave quadratic polynomial with respect to $y$, since $\frac{\partial^2F}{\partial y^2} = -\frac{1}{128}(4-x)^2 \le 0$. Solving $\frac{\partial F}{\partial y} = 0$ gives the interior critical path $y_* = \frac{x}{6(4-x)}$, which belongs to $[0,1]$ when $x \le \frac{24}{7}$. Substituting $y_*$ into $F(x,y)$ yields
\[
F(x,y_*) = 1 -\frac{9}{32}x +\frac{81}{4096}x^2.
\]
Differentiating with respect to $x$ shows that the critical point of $F(x,y_*)$ occurs outside the domain at $x = \frac{64}{9} \approx 7.11$. Thus, $F(x,y_*)$ decreases strictly for $x \in [0, 4]$. The global maximum occurs at the boundary corner $x=0$, which gives $F(0,0)=1$. Evaluation on the remaining boundaries $y=0$ and $y=1$ confirms that $F(x,y) \le 1$ across the entire compact parameter space.

To establish the sharpness of the upper bound rigorously, we consider the function $f_4(z) \in \mathcal{C}_{\mathcal{AP}}$ explicitly defined by the structural construction:
\begin{equation}\label{ext4}
f_4(z) = \int_0^z \exp\left( \int_0^s \frac{e^{t^3}\sqrt{1+t^3}-1}{t}\,dt \right)ds = z+\frac{1}{8}z^4+\cdots.
\end{equation}
This function corresponds to the choice of the 3-fold symmetric Schwarz function $\omega(z)=z^3$, which generates the Carath\'{e}odory parameters $c_1=0$, $c_2=0$, and $c_3=2$. Substituting these values into \eqref{po20} yields $a_2=0$ and $a_3=0$. Direct evaluation of \eqref{po18} under these initial coefficients confirms $T_{3,1}(f_4)=1$.

\medskip
\noindent\textbf{Lower bound.}
Using the inequality $\Re\xi \ge -|\xi| = -y$, \eqref{po21} implies $T_{3,1}(f) \ge H(x,y)$, where
\[
H(x,y) = 1 - \frac{9}{32}x + \frac{725}{36864}x^2 - \frac{1}{768}x(4-x)y - \frac{1}{256}(4-x)^2y^2.
\]
To minimize $H(x,y)$ over the compact region $[0,4] \times [0,1]$, we examine the partial derivative with respect to $y$:
\[
\frac{\partial H}{\partial y} = -\frac{1}{768}x(4-x) - \frac{1}{128}(4-x)^2y = -(4-x)\left[ \frac{x}{768} + \frac{(4-x)y}{128} \right].
\]
For any $x \in [0,4)$ and $y \in [0,1]$, the factor $(4-x)$ is positive and the expression inside the brackets is strictly positive except at the single corner $x=0, y=0$. Therefore, $\frac{\partial H}{\partial y} < 0$ holds uniformly on the interior of the domain. 

Because $\frac{\partial H}{\partial y}$ is strictly negative, $H(x,y)$ is a strictly decreasing function of $y$ for any fixed value of $x$. This guarantees that no interior local minimizers can exist inside the rectangle $(0,4) \times (0,1)$, meaning the absolute global minimum must reside on the boundary segment $y = 1$.

Setting $y = 1$, the function $H(x,y)$ simplifies to the single-variable quadratic polynomial:
\[
h(x) = H(x,1) = \frac{629}{36864}x^2 - \frac{49}{192}x + \frac{15}{16}.
\]
Differentiating $h(x)$ with respect to $x$ yields $h'(x) = \frac{629}{18432}x - \frac{49}{192}$. Setting $h'(x) = 0$ gives a critical point at $x_0 = \frac{4704}{629} \approx 7.48$. Since $x_0 > 4$, there are no local minimum turnarounds inside the interval $[0,4]$. Given that $h'(x) < 0$ for all $x \in [0,4]$, $h(x)$ decreases monotonically across the domain, forcing the absolute minimum to be attained at the rightmost endpoint $x = 4$.

Substituting $x = 4$ into $h(x)$ yields the global lower bound:
\[
T_{3,1}(f) \ge h(4) = 1 - \frac{9}{32}(4) + \frac{725}{36864}(16) = \frac{437}{2304}.
\]
This lower bound is sharp for the function $f_1(z)$ defined in \eqref{ext1}, which maps to the maximal Carath\'{e}odory parameters $c_1=2$ and $c_2=2$. Substituting these parameters into \eqref{po20} gives the explicit coefficient values $a_2 = \frac{3}{4}$, $a_3 = \frac{25}{48}$. Evaluating the determinant structure $T_{3,1}(f_1) = 2 a_2^2 a_3 - 2 a_2^2 - a_3^2 + 1$ with these values results in:
\[
T_{3,1}(f_1) = 2\left(\frac{9}{16}\right)\left(\frac{25}{48}\right) - 2\left(\frac{9}{16}\right) - \left(\frac{625}{2304}\right) + 1 = \frac{437}{2304},
\]
confirming the absolute sharpness of the lower bound.
\end{proof}

\section{Conclusion}
In this paper, we have provided a systematic investigation into several classical coefficient problems for the apple-like convex subclass $\mathcal{C}_{\mathcal{AP}}$, which is geometrically characterized by the analytical subordination relation $1+zf''(z)/f'(z) \prec e^z\sqrt{1+z}$. By exploiting the structural properties of Carath\'{e}odory functions alongside elegant parametric optimization methods, we successfully established sharp upper bounds for the initial inverse logarithmic coefficients $\Gamma_1$, $\Gamma_2$, and $\Gamma_3$. 

Furthermore, we extended our study to compound functionals by determining the precise, sharp upper and lower bounds for the consecutive coefficient difference $|\Gamma_2| - |\Gamma_1|$, as well as the sharp upper bound for the second-order inverse logarithmic Hankel determinant $H_{2,1}\left(F_{f^{-1}}/2\right)$. In addition, we fully characterized the generalized Fekete--Szeg\H{o} functional over real parameters through piecewise optimization curves and computed the sharp bounds for the third-order Hermitian--Toeplitz determinant $T_{3,1}(f)$.

Crucially, every analytical bound established in this study was proved to be sharp by identifying its corresponding extremal function. These results provide explicit, rigorous data points that enrich the structural theory of inverse coefficient mappings within non-homogeneous, asymmetric target domains. A comprehensive compilation of these sharp bounds and their matching extremal references is summarized in Table \ref{tab:extremal_summary}.

\begin{table}[H]
\centering
\caption{Summary of Sharp Bounds and Extremal Functions for $\mathcal{C}_{\mathcal{AP}}$.}
\label{tab:extremal_summary}
\small
\addtolength{\tabcolsep}{4pt} 
\renewcommand{\arraystretch}{1.5} 
\begin{tabular}{llcc}
\toprule
\textbf{Functional Structure} & \textbf{Type of Bound} & \textbf{Sharp Value} & \textbf{Extremal Source} \\ \midrule
$|\Gamma_1|$ & Upper Bound & $\dfrac{3}{8}$ & $f_1(z)$ defined in \eqref{ext1} \\[2mm]
$|\Gamma_2|$ & Upper Bound & $\dfrac{31}{192}$ & $f_1(z)$ defined in \eqref{ext1} \\[2mm]
$|\Gamma_3|$ & Upper Bound & $\dfrac{205}{2304}$ & $f_1(z)$ defined in \eqref{ext1} \\[2mm]
$|\Gamma_2| - |\Gamma_1|$ & Upper Bound & $\dfrac{1}{8}$ & $f_2(z)$ defined in \eqref{ext2} \\[2mm]
$|\Gamma_2| - |\Gamma_1|$ & Lower Bound & $-\dfrac{3\sqrt{6}}{4\sqrt{55}}$ & $f_3(z)$ defined in \eqref{ext3} \\[2mm]
$\left| H_{2,1}\left(F_{f^{-1}} / 2\right)\right|$ & Upper Bound & $\dfrac{3413}{195328}$ & $f_*(z)$ defined in \eqref{ext_hankel} \\[2mm]
$|a_3 - \lambda a_2^2| - \mu |a_2|$ & Upper Bound (Constant) & $\dfrac{1}{4}$ & $f_2(z)$ defined in \eqref{ext2} \\[2mm]
$|a_3 - \lambda a_2^2| - \mu |a_2|$ & Upper Bound (Variable) & $\dfrac{1}{48} |25 - 27\lambda| - \dfrac{3\mu}{4}$ & $f_1(z)$ defined in \eqref{ext1} \\[2mm]
$T_{3,1}(f)$ & Upper Bound & $1$ & $f_4(z)$ defined in \eqref{ext4} \\[2mm]
$T_{3,1}(f)$ & Lower Bound & $\dfrac{437}{2304}$ & $f_1(z)$ defined in \eqref{ext1} \\ \bottomrule
\end{tabular}
\end{table}

\section*{Declarations}
\subsection*{Funding}
The first author acknowledge financial support from the Council of Scientific and Industrial Research (CSIR), New Delhi, India, under Grant Nos. 09/1224(16975)/2023-EMR-I.
\subsection*{Data Availability Statement}
Data sharing is not applicable to this article as no datasets were generated or analyzed during the current study.
\subsection*{Conflict of Interest}
The authors declare that they have no conflict of interest. 
\subsection*{Author Contributions}
Both authors contributed equally to this work.

\end{document}